\documentclass{article}
\usepackage{arxiv}

\usepackage[utf8]{inputenc}
\usepackage[T1]{fontenc}
\usepackage{hyperref}
\usepackage{url}
\usepackage{amsmath,amssymb,amsthm,amsfonts}
\usepackage{nicefrac}
\usepackage{microtype}
\usepackage{graphicx}
\usepackage{mathtools,bm,mathrsfs}
\usepackage{enumitem}
\usepackage{booktabs}
\usepackage{array}

\newtheorem{theorem}{Theorem}[section]
\newtheorem{lemma}[theorem]{Lemma}
\newtheorem{proposition}[theorem]{Proposition}

\theoremstyle{remark}

\newcommand{\T}{\mathcal T}
\newcommand{\SigmaN}{\Sigma_N}
\newcommand{\Csafe}{\mathcal C}
\newcommand{\Eaug}{\mathscr E}
\newcommand{\Daug}{\mathscr D}
\newcommand{\Flam}{\mathfrak F}
\newcommand{\esssup}{\operatorname*{ess\,sup}}
\newcommand{\diag}{\operatorname{diag}}

\newcommand{\divv}{\operatorname{div}}
\newcommand{\R}{\mathbb R}
\newcommand{\fint}{\mathop{\ooalign{$\int$\cr\kern0.32em--\cr}}}

\newcommand{\subjclass}[2][2020]{%
  \par\addvspace\medskipamount
  {\small\textit{2020 Mathematics Subject Classification.} #2\par}}

\hypersetup{
  pdftitle  = {Quantitative vanishing-regularization limits for degenerate cross-diffusion with changing boundary strata},
  pdfauthor = {Shchetinin, E.Yu. and Shevchuk, A.A. and Salpagarov, S.I.},
  pdfsubject = {math.AP},
  pdfkeywords = {cross-diffusion, degenerate parabolic systems, relative entropy, singular limit, intrinsic De Giorgi iteration, weak-strong stability}
}

\title{Quantitative vanishing-regularization limits for\\
       degenerate cross-diffusion with changing boundary strata}

\author{%
  E.\,Yu.~Shchetinin \\
  Sevastopol State University \\
  Sevastopol, Russia \\
  \texttt{riviera-molto@mail.ru}
  \And
  A.\,A.~Shevchuk \\
  Sevastopol State University \\
  Sevastopol, Russia \\
  \texttt{andreiluck11@yandex.ru}
  \AND
  S.\,I.~Salpagarov \\
  RUDN University \\
  Moscow, Russia \\
  \texttt{salpagarov-si@rudn.ru}
}

\begin{document}
\maketitle

\begin{abstract}
We study a vanishing-regularization limit for a degenerate cross-diffusion system on the simplex when the active boundary stratum of a smooth reference solution may change in space and time. The mobility
\[
M(u)=\diag u-u\otimes u
\]
loses rank when a component vanishes, so a fixed relative entropy or an unweighted two-point gradient estimate cannot remain uniformly coercive across moving faces. We introduce a geometric-relative framework that separates the single emerging normal direction from the active tangential subsystem. Its ingredients are a global resolvent geodesic distance, a boundary-oriented rank-one perspective correction, an intrinsic Schur-normal De Giorgi alternative, a mean-free tangential iteration with temporal-flux control, and a terminal coupled Schur compactness argument. These mechanisms yield a non-circular first-exit theorem for smooth approximations and then pass to admissibly approximable entropy weak solutions. If the strong solution stays a fixed distance from the codimension-two skeleton and its vanishing components satisfy
\[
|\nabla v_i|\lesssim v_i^{\alpha},\qquad
|\{0<v_i<r\}|\lesssim r^{\sigma},
\]
then the augmented relative energy is bounded by
\[
\rho_\lambda=\begin{cases}
\lambda^{2\alpha+\sigma},&2\alpha+\sigma<2,\\
\lambda^2|\log\lambda|,&2\alpha+\sigma=2,\\
\lambda^2,&2\alpha+\sigma>2.
\end{cases}
\]
For a hypersurface contact $v_i\asymp d^p$, $p\ge2$, this gives
$\rho_\lambda=O(\lambda^{2-1/p})$ and
$\|u^\lambda-v\|_{L^\infty(0,T;L^2)}=O(\lambda^{1-1/(2p)})$.
\end{abstract}


% MSC 2020: 35K51, 35K55, 35B40, 35A02

\section{Introduction}

Cross-diffusion systems couple the flux of each component to gradients of several components. Their diffusion matrices are frequently nonsymmetric or degenerate, and standard uniformly parabolic arguments are therefore unavailable. Entropy variables and boundedness-by-entropy methods provide a robust existence framework for many such systems; see, among others, \cite{Jungel2015,Jungel2016}. Relative entropy has also become a central tool for weak--strong uniqueness and stability in cross-diffusion, including reaction--cross-diffusion systems \cite{ChenJungel2019}, Maxwell--Stefan systems \cite{HuoJungelTzavaras2022}, and nonlocal or volume-filling models \cite{JungelPortischZurek2022,Liu2025,HeitzingerJungel2025}. Partial regularity based on entropy-adapted Campanato methods was developed in \cite{BraukhoffRaithelZamponi2022}; a recent glued-entropy approach substantially weakens positivity assumptions in weak--strong uniqueness \cite{GeltnerHeitzingerJungel2026}.

The present paper addresses a different difficulty: the \emph{active boundary stratum of the reference solution is allowed to move}. We consider the simplex
\[
\SigmaN=\{u\in[0,1]^N:\mathbf 1\cdot u=1\},\qquad
\T=\{\xi\in\R^N:\mathbf 1\cdot\xi=0\},
\]
and the covariance mobility
\[
M(u)=\diag u-u\otimes u.
\]
If one component vanishes, $M(u)$ acquires an additional kernel direction. Hence the geometry itself changes with the state. A relative functional that is uniformly coercive on one fixed face need not remain uniformly coercive while the strong solution moves from the interior to that face or back.

The model regularization is
\begin{align}
\partial_tu^\lambda-\divv\big((M(u^\lambda)+\lambda I_{\T})\nabla u^\lambda\big)&=0,
\label{eq:reg}\\
\partial_tv-\divv\big(M(v)\nabla v\big)&=0,
\label{eq:limit}
\end{align}
with no-flux boundary conditions and common initial data. The parameter $\lambda>0$ removes the rank defect, but the coercivity degenerates as $\lambda\downarrow0$ precisely in the normal direction to a face. Our objective is a quantitative limit that remains valid when the vanishing component of $v$ depends on $(t,x)$.

The analysis is not a direct consequence of a standard weak--strong estimate. Several natural stronger claims are false. The ordinary affine Hessian of the squared geodesic distance is not uniformly controlled in the boundary layer; the unweighted normal gradient dissipation can vanish like a power of $\lambda$; global matrix-weighted two-point coercivity fails for safe pairs with a large tangential mismatch; and a critical Morrey bound for the Schur cross flux does not by itself imply pointwise tangential smallness. These obstructions dictate an anisotropic construction.

\subsection*{Main ideas}
The proof separates the local geometry into one normal coordinate and an active tangential block.

First, we use the resolvent metric
\[
g_{\lambda,z}(\xi,\zeta)=\xi^\top(M(z)+\lambda I_{\T})^{-1}\zeta.
\]
The corresponding kinetic action is jointly convex in state and velocity. This gives a unique minimizing geodesic, and a Fisher--Rao comparison keeps geodesics between safe endpoints away from the codimension-two skeleton. The squared distance is therefore globally well defined on the safe region.

Second, the geodesic functional is augmented only in the emerging normal species. The profile
\[
\Phi''(R)=\begin{cases}R^{-1},&R\le1,\\R^{-2},&R\ge1,\end{cases}
\]
produces a boundary-oriented perspective divergence. It removes both separated normal-ratio singularities while leaving the active tangential block uniformly coercive. The only pure strong-state remainder is the regularizing flux $\lambda\nabla v_i$.

Third, in constant affine Schur coordinates the local diffusion matrix has the form
\[
\mathcal A_\lambda(r,z)=
\begin{pmatrix}
 a_\lambda & b^\top\\ b&C_\lambda
\end{pmatrix},
\qquad
a_\lambda\asymp\lambda+r,\quad |b|\lesssim r,\quad C_\lambda\succeq cI.
\]
The normal equation admits a $\lambda/k$-uniform intrinsic De Giorgi alternative. The tangential subsystem is handled by a mean-free Campanato iteration: one subtracts the spatial mean at each time, while the evolution of that mean is controlled separately by a local flux identity. This removes the temporal constant mode that survives when the intrinsic mobility scale tends to zero.

Finally, at the terminal intrinsic scale the full Schur system is retained. The normal-to-tangential feedback carries the small coefficient $k^2/(\lambda+k)$. In the singular limit the mean-free tangential profile becomes spatially harmonic, whereas the normal profile solves a scalar uniformly parabolic equation with a small harmonic-gradient drift. The drift is divergence free, which yields an exact truncation cancellation and a uniform De Giorgi oscillation decay. The non-singular terminal regime is closed by the same triangular compactness scheme.

These ingredients lead to a smooth simultaneous first-exit theorem. The passage to entropy weak solutions is then a compactness/lower-semicontinuity step rather than a new regularity argument.

\subsection*{Contribution relative to existing entropy methods}
The purpose of the paper is not to replace the general existence theory for cross-diffusion. Instead, it isolates a quantitative mechanism for a rank-changing singular limit. Glued entropies such as those in \cite{BraukhoffRaithelZamponi2022,GeltnerHeitzingerJungel2026} are designed to regularize degeneracies in state-space entropy variables. Here the difficulty is more geometric: one normal mode emerges and disappears with the active face, while the active tangential subsystem remains uniformly parabolic. The resulting estimate is therefore intentionally anisotropic. In particular, we do \emph{not} claim a global unweighted estimate for $\nabla(u^\lambda-v)$.

The paper is organized as follows. Section~\ref{sec:setting} states the hypotheses and the main theorem. Section~\ref{sec:geometry} records the geodesic and perspective functionals. Section~\ref{sec:atlas} fixes the synchronized face/interior atlas. Sections~\ref{sec:affine} and \ref{sec:normal} prove the augmented affine estimate and the normal intrinsic alternative. Sections~\ref{sec:meanfree}--\ref{sec:terminal} establish the tangential mean-free iteration and terminal coupled regularity. The global smooth continuation and cutoff gluing are proved in Section~\ref{sec:global}. Section~\ref{sec:weak} gives the weak passage and Section~\ref{sec:rates} the contact-dependent convergence rates.

\section{Setting and main results}\label{sec:setting}

Let $N\ge3$, let $\Omega\subset\R^d$ be a bounded $C^{1,1}$ domain, and set $Q_T=(0,T)\times\Omega$. We order components increasingly,
$z_{[1]}\le\cdots\le z_{[N]}$, and define
\[
\Csafe_a=\{z\in\SigmaN:z_{[2]}\ge a\}.
\]
Thus no more than one component can be smaller than $a$ on $\Csafe_a$.

We impose the following hypotheses on the reference solution $v$.

\begin{enumerate}[label=(H\arabic*)]
\item \label{Hsmooth} $v\in C^2([0,T]\times\overline\Omega;\SigmaN)$ solves \eqref{eq:limit} with no-flux boundary condition.
\item \label{Hsafe} There exists $\eta>0$ such that
$v([0,T]\times\overline\Omega)\subset\Csafe_\eta$.
\item \label{Hcontact} For every $i$ there are $K,r_c>0$, $\alpha\ge0$, $\sigma>0$ such that
\[
|\nabla v_i|\le K v_i^\alpha\quad\text{on }\{0<v_i<r_c\},
\]
\[
|\{(t,x)\in Q_T:0<v_i(t,x)<r\}|\le Kr^\sigma,
\qquad0<r<r_c.
\]
\item \label{Hinitial} $u^\lambda(0)=v(0)$.
\end{enumerate}

For fixed $\lambda>0$, an entropy weak solution is called \emph{admissibly approximable} if it is the strong $L^2(Q_T)$ and a.e. limit of smooth simplex-valued no-flux approximations with the common initial trace, while the gradients converge weakly in $L^2(Q_T)$ and the standard entropy bounds are uniform along the approximation.

The forcing scale is
\begin{equation}
\Flam_\lambda[v;T]
:=\lambda^2\sum_{i=1}^N\int_{Q_T}
\frac{|\nabla v_i|^2}{(v_i+\lambda)^2}\,dxdt.
\label{eq:forcing}
\end{equation}

\begin{theorem}[Quantitative weak changing-strata limit]\label{thm:main}
Assume \ref{Hsmooth}--\ref{Hinitial}. Let $u^\lambda$ be an admissibly approximable entropy weak solution of \eqref{eq:reg}. Then there exist $\lambda_0>0$ and constants $c,C>0$, independent of $0<\lambda\le\lambda_0$, such that
\begin{equation}
 u^\lambda(t,x)\in\Csafe_{\eta/4}
 \quad\text{for a.e. }(t,x)\in Q_T,
 \label{eq:safe-main}
\end{equation}
and
\begin{equation}
\esssup_{0\le t\le T}\Eaug_\lambda(u^\lambda(t)\mid v(t))
+c\int_0^T\Daug_\lambda(t)\,dt
\le C\Flam_\lambda[v;T].
\label{eq:mainenergy}
\end{equation}
Moreover
\begin{equation}
\esssup_{0\le t\le T}
\|u^\lambda(t)-v(t)\|_{L^2(\Omega)}^2
\le C\Flam_\lambda[v;T].
\label{eq:l2main}
\end{equation}
The dissipation $\Daug_\lambda$ is anisotropic: on a face chart it controls the active tangential gradient error uniformly and the unique normal gradient error with a state-dependent weight. No global unweighted estimate for the full gradient error is asserted.
\end{theorem}

\begin{theorem}[Contact-dependent rate]\label{thm:rate}
Under \ref{Hcontact},
\[
\Flam_\lambda[v;T]\le C\rho_\lambda,
\qquad
\rho_\lambda=
\begin{cases}
\lambda^{2\alpha+\sigma},&2\alpha+\sigma<2,\\
\lambda^2|\log\lambda|,&2\alpha+\sigma=2,\\
\lambda^2,&2\alpha+\sigma>2.
\end{cases}
\]
If locally near a smooth moving zero interface
$v_i\asymp d^p$ and $|\nabla v_i|\asymp d^{p-1}$ with $p\ge2$, then
\[
\rho_\lambda=O(\lambda^{2-1/p}),
\qquad
\|u^\lambda-v\|_{L^\infty(0,T;L^2(\Omega))}
=O\bigl(\lambda^{1-1/(2p)}\bigr).
\]
\end{theorem}

\section{Geodesic functional and boundary-oriented perspective correction}\label{sec:geometry}

For $0<\lambda\le1$ define
\[
g_{\lambda,z}(\xi,\zeta)=\xi^\top(M(z)+\lambda I_{\T})^{-1}\zeta,
\qquad \xi,\zeta\in\T.
\]
Let $d_{g_\lambda}$ be the associated distance and
\[
\mathcal G_\lambda(u,v)=\frac12d_{g_\lambda}(u,v)^2.
\]
The action Lagrangian
\[
L_\lambda(z,\xi)=\frac12\xi^\top(M(z)+\lambda I_{\T})^{-1}\xi
\]
is jointly convex. Writing $A_\lambda=M+\lambda I_{\T}$ and $y=A_\lambda^{-1}\xi$, one obtains
\begin{equation}
D^2L_\lambda[(h,k)]^2
=|k-DA_\lambda[h]y|_{A_\lambda^{-1}}^2+(h\cdot y)^2\ge0.
\label{eq:jointconvex}
\end{equation}
For variations with fixed endpoints the equality case reduces to a linear ODE with zero initial data, which gives uniqueness of the minimizing action. A shifted Fisher comparison prevents minimizing curves between endpoints in fixed safe classes from approaching the codimension-two skeleton. We use this only to guarantee a globally unambiguous two-point state functional; the local PDE coercivity below is not deduced from a global affine Hessian bound.

Since $M(z)\preceq I_{\T}$ on the simplex,
$M(z)+\lambda I_{\T}\preceq2I_{\T}$. Hence
\begin{lemma}[State coercivity]\label{lem:statecoerc}
For all $u,v\in\SigmaN$ and $0<\lambda\le1$,
\[
\mathcal G_\lambda(u,v)\ge\frac14|u-v|^2.
\]
\end{lemma}

The normal correction uses
\begin{equation}
\Phi(R)=
\begin{cases}
R\log R-R+1,&0<R\le1,\\
R-1-\log R,&R\ge1,
\end{cases}
\qquad
\Phi''(R)=
\begin{cases}
R^{-1},&R\le1,\\
R^{-2},&R\ge1.
\end{cases}
\label{eq:Phi}
\end{equation}
Thus $(v_i+\lambda)\Phi((u_i+\lambda)/(v_i+\lambda))$ is the shifted Kullback--Leibler divergence oriented toward the smaller normal scale. Its Hessian in the two endpoint variables has rank one.

The exact perspective identity is
\begin{align}
\frac d{dt}\int V\Phi(U/V)
&=-\int\Phi''(R)\nabla R\cdot(J_i^\lambda-RJ_i^0),\label{eq:perspective}\\
U&=u_i+\lambda,\qquad V=v_i+\lambda,\qquad R=U/V.\notag
\end{align}
The identity uses the full $i$-th row flux and is the reason no derivative of a state-space cutoff enters the singular normal algebra.

\section{Synchronized face/interior atlas}\label{sec:atlas}

Fix a face $i$ and write
\[
\pi_i(u)=\frac{u_{-i}}{1-u_i}.
\]
Let $r_A>0$ and $\delta_A>0$ be the radius and tangential width for the local affine estimate below. Choose constants in the order
\begin{equation}
0<s_b\ll k_b\ll r_A\ll\delta_\tau\le\frac14\delta_A,
\qquad r_A\ll\eta.
\label{eq:hierarchy}
\end{equation}
More explicitly, after fixing the universal normal De Giorgi constant $\delta\in(0,1)$ and $\kappa=1-\delta$, we choose
\[
k_b=c_kr_A,
\qquad
s_b\le c_sk_b,
\qquad
c_s<\frac{\delta\kappa}{16}.
\]
We also require $8s_b<k_b<r_A/8<\eta/64$.

Choose cutoffs $\chi_i\in C^2([0,1];[0,1])$ with
\[
\chi_i(s)=1\quad(s\le r_A/2),
\qquad
\chi_i(s)=0\quad(s\ge3r_A/4).
\]
Because the strong state belongs to $\Csafe_\eta$, at most one such face cutoff is active at a point. Define the global augmented energy
\begin{equation}
\Eaug_\lambda(u\mid v)=
\int_\Omega\left[
\mathcal G_\lambda(u,v)
+\varepsilon\sum_{i=1}^N
\chi_i(v_i)(v_i+\lambda)
\Phi\left(\frac{u_i+\lambda}{v_i+\lambda}\right)
\right]dx.
\label{eq:Eaug}
\end{equation}

The atlas is organized so that the singular face core and the interior patch overlap only after the regularized state has acquired a fixed normal lower bound. This avoids the inconsistent hierarchy in which a cutoff would be active outside the domain of the local face estimate.

\begin{lemma}[Face/interior handoff]\label{lem:handoff}
After decreasing the fixed constants in \eqref{eq:hierarchy}, the same intrinsic normal alternative used in Section~\ref{sec:normal} has the following consequences.
\begin{enumerate}[label=(\roman*)]
\item If $0\le v_i\le2s_b$, a first exit through $u_i=k_b$ is impossible while the tangential tube is valid.
\item If $2s_b\le v_i\le r_A/2$ and $u_i$ reaches $r_A$, the upper normal branch is incompatible with the center value; hence a child cylinder satisfies $u_i\ge\delta r_A$.
\item If $u_i<r_A$ until $v_i$ reaches $r_A/2$, a fixed finite chain of normal alternatives produces a lower bound $u_i\ge c_Ir_A$ before the perspective cutoff switches off.
\end{enumerate}
\end{lemma}

\begin{proof}
The proof is a finite-depth application of Lemma~\ref{lem:DG}. In the face core, a lower branch would give $u_i\ge\delta k_b$ on a child cylinder while $v_i\le2s_b<\delta k_b/4$. The perspective density then has a fixed positive lower bound, contradicting small augmented energy. At the outer barrier $u_i=r_A$, the upper branch gives $u_i\le\kappa r_A$ at the center and is impossible, so the lower branch creates the desired interior-entry bound. Finally, iterate the finite levels $r_A,\kappa r_A,\ldots,c_Ir_A$. If all branches are upper, $u_i\le c_Ir_A$ while $v_i\ge r_A/2$, again contradicting the augmented state energy for $c_I$ sufficiently small.
\end{proof}

\section{Uniform augmented affine estimate}\label{sec:affine}

Use constant affine Schur coordinates near the face $u_i=0$ and split a spatial error gradient into the single normal channel $E_n$ and the active tangential block $E_\tau$.

\begin{proposition}[Full-block affine closure]\label{prop:affine}
There exist fixed $r_A,\delta_A,\varepsilon,c,C>0$, independent of $\lambda$, such that if
\[
u_i,v_i\le r_A,
\qquad
|\pi_i(u)-\pi_i(v)|\le\delta_A,
\]
then the local augmented density
\[
\mathcal H_{\lambda,i}(u,v)
=\mathcal G_\lambda(u,v)
+\varepsilon(v_i+\lambda)
\Phi\left(\frac{u_i+\lambda}{v_i+\lambda}\right)
\]
satisfies for smooth solutions
\begin{align}
\frac d{dt}\int_\Omega\mathcal H_{\lambda,i}
&+c\int_\Omega|E_\tau|^2
+c\int_\Omega\omega_i(u,v)|E_n|^2\notag\\
&\le C\int_\Omega\mathcal H_{\lambda,i}|\nabla v|^2
+C\lambda^2\int_\Omega\frac{|\nabla v_i|^2}{(v_i+\lambda)^2},
\label{eq:affine}
\end{align}
up to the standard localization terms, where
\[
\omega_i(u,v)=\min\left\{1,\frac{v_i+\lambda}{u_i+\lambda}\right\}.
\]
No normal-ratio corridor is required.
\end{proposition}

\begin{proof}
The perspective part is given by \eqref{eq:perspective}. On the scalar normal fiber, optimization in the error derivative gives the exact positive remainders
\[
\mathcal R_-=
\frac{(\lambda+u_i)[b_0(v_i)-b_\lambda(u_i)]^2}
{4(\lambda+v_i)^2b_\lambda(u_i)}
\quad(R\le1),
\]
\[
\mathcal R_+=
\frac{[b_0(v_i)-b_\lambda(u_i)]^2}
{4(\lambda+v_i)b_\lambda(u_i)}
\quad(R\ge1),
\]
which yield the weight $\omega_i$ and the forcing scale in \eqref{eq:affine}. To justify uniformity of the full block, argue by contradiction with a normalized violating sequence
\[
|E_{\tau,m}|^2+\omega_{i,m}|E_{n,m}|^2=1
\]
while $r_A,\delta_A\downarrow0$. If the normal ratio stays in a compact subset of $(0,\infty)$, the two scales are comparable and the compact-corridor affine estimate applies. If the ratio tends to $0$ or $\infty$, the scalar normal form converges after its natural normalization to one of the two positive forms above. At $r_A=0$ the Schur matrix is block diagonal and the active face block is strictly positive. The mixed blocks therefore vanish in the limit, and for small $r_A+\delta_A$ are absorbed by Young's inequality.

For the strong-state terms use the exact row decomposition
\begin{align*}
J_i^\lambda-RJ_i^0
={}&[\lambda e_i^\top+u_i\mathcal A_i(u)]\nabla(u-v)+\lambda\nabla v_i\\
&+[u_i\mathcal A_i(u)-Rv_i\mathcal A_i(v)]\nabla v.
\end{align*}
The last bracket always contains a state error, since
\[
|u_i-Rv_i|=\frac{\lambda|u_i-v_i|}{v_i+\lambda}\le|u_i-v_i|.
\]
It is therefore bounded by $C\mathcal H_{\lambda,i}|\nabla v|^2$ after Young's inequality. The only pure strong-state term is $\lambda\nabla v_i$, giving the final term in \eqref{eq:affine}.
\end{proof}

\section{Uniform Schur-normal intrinsic alternative}\label{sec:normal}

In the same coordinates the normal equation reads
\begin{equation}
r_t-\divv(a_\lambda(r,z)\nabla r+b(r,z)^\top\nabla z)=0,
\label{eq:normal}
\end{equation}
with
\[
c(\lambda+r)\le a_\lambda\le C(\lambda+r),
\qquad |b(r,z)|\le Cr.
\]
For
\[
Q_R^k=(t_0-R^2/(\lambda+k),t_0)\times B_R(x_0)
\]
set
\[
\mathfrak B_{R,k}
=\frac{R^2}{k(\lambda+k)}\|\nabla z_v\|_{L^\infty(Q_R^k)}^2,
\qquad
\mathfrak R_{R,k}
=\frac1{kR^d}\int_{Q_R^k}|\nabla(z-z_v)|^2.
\]

\begin{lemma}[Uniform intrinsic De Giorgi alternative]\label{lem:DG}
There exist universal $\delta\in(0,1)$, $\widehat\theta\in(0,1/4)$ and $\varepsilon_0>0$ such that if
\[
0\le r\le k\le r_A,
\qquad
\mathfrak B_{R,k}+\mathfrak R_{R,k}\le\varepsilon_0,
\]
then on the concentric intrinsic child of radius
$R'=\widehat\theta\sqrt{1-\delta}\,R$ one has
\[
r\ge\delta k
\qquad\text{or}\qquad
r\le(1-\delta)k.
\]
The constants are independent of $a=\lambda/k\in[0,\infty)$ and the same statement holds on one-sided initial and no-flux boundary half-cylinders.
\end{lemma}

\begin{proof}
For $w=(\ell-r)_+$ the normal Caccioppoli estimate is
\begin{align}
&\sup_t\int\eta^2w^2
+c\iint\eta^2(\lambda+r)|\nabla w|^2\notag\\
&\le C\iint w^2\bigl(|\eta_t|+(\lambda+k)|\nabla\eta|^2\bigr)\notag\\
&\quad+C\ell\iint_{\{r<\ell\}}\eta^2|\nabla(z-z_v)|^2
+C\ell|Q_R^k|\|\nabla z_v\|_\infty^2.\label{eq:caccnormal}
\end{align}
The key estimate is
\[
\frac{|b^\top\nabla z|^2}{\lambda+r}
\le C r|\nabla z|^2,
\]
which produces the level factor $\ell$.

Rescale $r=kq$, $x=x_0+Ry$, $t=t_0+R^2(\lambda+k)^{-1}s$. The diffusion coefficient is comparable to
\[
\alpha_a(q)=\frac{a+q}{a+1}.
\]
Introduce
\[
\Psi_a(q)=\int_0^q\sqrt{\frac{a+s}{a+1}}\,ds.
\]
Then $|\nabla\Psi_a(q)|^2=\alpha_a(q)|\nabla q|^2$, while uniformly for $a\ge0$,
\[
\Psi_a(1/2)-\Psi_a(1/4)\ge c_*>0.
\]
Thus the parabolic De Giorgi isoperimetric argument applied to transformed truncations has constants independent of $a$. If the upper-level set is small, the iteration takes place where $q\ge1/4$ and is uniformly elliptic, giving the upper alternative. Otherwise the expansion-of-positivity argument applied to lower truncations yields the lower alternative. The forcing always enters multiplied by the current level and by $\mathfrak B_{R,k}+\mathfrak R_{R,k}$, so a universal smallness threshold preserves the recurrence. The initial and boundary variants use the zero initial relative trace and the exact no-flux cancellation in constant Schur coordinates.
\end{proof}

\section{Mean-free tangential iteration and temporal-flux control}\label{sec:meanfree}

Let $w=z-z_v$. On an upper normal chain set
\[
k_{n+1}=\kappa k_n,
\quad R_{n+1}=\theta R_n,
\quad \mu_n=\lambda+k_n,
\quad 0<\theta<\kappa/8,
\]
with intrinsic cylinders $Q_n$. At each time define the local spatial mean
\[
m_n(t)=\fint_{B_{R_n}}w(t,x)\,dx
\]
and the mean-free excess
\begin{equation}
Y_n=\frac1{\mu_n^2}\fint_{Q_n}|w-m_n(t)|^2.
\label{eq:Yn}
\end{equation}
The time-dependent spatial constants are deliberately removed from $Y_n$.

\begin{lemma}[Mean-free contraction]\label{lem:meanfree}
For sufficiently small fixed atlas constants there exist $q_0\in(0,1)$ and $C>0$, independent of $0<\mu_n\le1$, such that on every upper step
\begin{equation}
Y_{n+1}\le q_0Y_n+C(k_n^2+S_n^2+R_n^2),
\qquad
S_n=\frac{R_n}{\mu_n}K_v.
\label{eq:Ycontract}
\end{equation}
Consequently
\[
\sum_{n\ge0}\sqrt{Y_n}<\infty
\]
uniformly in the number of upper steps and in $\lambda$.
\end{lemma}

\begin{proof}
After intrinsic scaling and the normalization $w=c_n+\mu_n\zeta$, the tangential difference equation has the schematic form
\[
\mu_n\partial_\tau\zeta-\divv(C_n\nabla\zeta)=\divv G_n,
\]
where
\[
\|G_n\|_{L^2(Q_1)}
\le C(k_n+S_n+R_n+\varepsilon_c\sqrt{Y_n}).
\]
Subtract the spatial mean of $\zeta$ at each time. Testing by a cutoff times this mean-free part eliminates the time derivative of the weighted mean exactly. Spatial difference quotients and Poincar\'e--Campanato then give a contraction whose constants do not use the coefficient in front of $\partial_\tau$. The Schur cross term contributes the small structural factor
\[
\frac{k_n^2}{\mu_n}\le k_n.
\]
Choosing the child scale first and then the fixed atlas constants gives \eqref{eq:Ycontract}. Since $k_n$, $R_n$, and $S_n$ decay geometrically, the errors are Dini summable.
\end{proof}

The removed temporal mean is controlled by the local flux, not by a false local conservation law. Let $\chi_n$ be a spatial cutoff for $B_{R_n}$ and define the weighted mean
\[
\widehat m_n(t)=\frac{\int\chi_n w(t)\,dx}{\int\chi_n\,dx}.
\]

\begin{lemma}[Temporal mean-flux estimate]\label{lem:meanflux}
On the upper chain,
\[
\operatorname*{osc}_{t\in I_{n+1}}\widehat m_n(t)
\le C(\sqrt{Y_n}+k_n+S_n+R_n),
\]
and at the terminal time
\[
|\widehat m_{n+1}(t_0)-\widehat m_n(t_0)|
\le C(\sqrt{Y_n}+\sqrt{Y_{n+1}}+k_n+S_n+R_n).
\]
Hence the mean jumps are absolutely summable.
\end{lemma}

\begin{proof}
In normalized variables the weighted mean of $\zeta$ satisfies the exact flux balance
\[
\mu_n\frac d{d\tau}\widehat\zeta_n
=-\frac{\int(C_n\nabla\zeta+G_n)\cdot\nabla\chi\,dy}{\int\chi\,dy}.
\]
The physical mean is $c_n+\mu_n\widehat\zeta_n$, so the singular factor cancels. Cauchy--Schwarz, the mean-free Caccioppoli estimate, and the forcing bound give the oscillation estimate. Comparing two nested weighted means on a common time slice and transporting both to $t_0$ yields the jump estimate; summability follows from Lemma~\ref{lem:meanfree}.
\end{proof}

\section{Terminal coupled Schur regularity}\label{sec:terminal}

The upper chain terminates either in an annular branch $\delta k\le r\le k$ or when $k\le\lambda$. Put $\mu=\lambda+k$, normalize
\[
r=kp,
\qquad
w=m(\tau)+\mu\zeta,
\qquad
x=x_0+Ry,
\qquad
t=t_0+R^2\mu^{-1}\tau.
\]
The principal system is
\begin{align}
p_\tau-\divv(A(p)\nabla p+B^\top\nabla\zeta)&=\divv F_p,
\label{eq:terminalp}\\
\mu\zeta_\tau-\divv(\varepsilon B\nabla p+C\nabla\zeta)&=\divv F_\zeta,
\label{eq:terminalz}
\end{align}
where
\[
A\ge c>0,
\qquad C\succeq cI,
\qquad
\varepsilon=\frac{k^2}{\mu}\le k.
\]

\begin{theorem}[Terminal coupled $\varepsilon$-regularity]\label{thm:terminal}
There are universal $k_c,\varepsilon_c>0$, $\vartheta\in(0,1/4)$ and $\gamma\in(0,1)$ such that, if a terminal cylinder lies inside the pre-exit tangential tube, $k\le k_c$, the mean-free tangential excess is at most $\varepsilon_c^2$, and the scaled strong-state errors are small, then
\begin{equation}
\mathcal P(\rho)+\mathcal Y(\rho)
\le C\rho^{2\gamma}\bigl(\mathcal P(1)+\mathcal Y(1)\bigr)
+C\omega_v(R)^2\rho^{2\gamma},
\label{eq:terminaldecay}
\end{equation}
where
\[
\mathcal P(\rho)=\fint_{Q_\rho}|p-(p)_{Q_\rho}|^2,
\qquad
\mathcal Y(\rho)=\fint_{Q_\rho}|\zeta-(\zeta)_{B_\rho}(\tau)|^2.
\]
The constants are uniform in both terminal branches and the corresponding weighted means of $w$ have a Dini-summable tail. Initial one-sided and $C^{1,1}$ no-flux boundary variants hold with geometric constants.
\end{theorem}

\begin{proof}
We record the compactness mechanism because it is the point at which the singular temporal mode must be treated correctly. Consider first a sequence with $\mu_j\to0$. Projection onto spatial mean-zero tangential tests eliminates the time-dependent constant mode. The term $\mu_j\zeta_{j,\tau}$ vanishes distributionally and the reverse coupling $\varepsilon_jB_j\nabla p_j$ tends to zero. Hence, for a.e. normalized time,
\[
-\divv(C_0(\tau)\nabla\zeta)=0.
\]
Since $C_0$ acts in target state space and is invertible, $\zeta$ is componentwise harmonic in the spatial variable (with the corresponding homogeneous Neumann condition on a flattened physical boundary).

The normal sequence is compact in $L^2_{\rm loc}$ by Aubin--Lions. Its limit satisfies
\[
p_\tau-\divv(A_*(p)\nabla p-\beta p\nabla h)=0,
\qquad
h=z_0(\tau)\cdot\zeta,
\]
with $\beta=\sqrt{c_N}$. The harmonic estimate and the small terminal tangential excess give
\[
\|\nabla h\|_{L^2_\tau L^\infty_y}\le C\varepsilon_c.
\]
Moreover $\Delta h=0$, so $b=\beta\nabla h$ is divergence free and tangent to a physical no-flux boundary. If $v=(p-\ell)_+$, then
\[
\int b\cdot\nabla p\,\eta^2v
=\frac12\int b\cdot\nabla(v^2)\eta^2
=-\int\eta v^2b\cdot\nabla\eta.
\]
Thus the drift contributes only multiplicatively to the truncation energy and creates no additive critical forcing floor. The scalar De Giorgi recurrence retains its superlinear exponent, yielding uniform normal oscillation decay. Harmonic Campanato decay yields the tangential part of \eqref{eq:terminaldecay}.

If $\liminf\mu_j>0$, argue by the same triangular compactness rather than by a general regularity theorem for nonlinear systems. In the contradiction sequence $k_j\to0$, so $k_j^2/\mu_j\to0$. The tangential limit is uniformly parabolic; the vanishing terminal mean-free threshold implies $\nabla\zeta_j\to0$ locally. The normal limit is therefore a scalar uniformly parabolic equation and standard scalar De Giorgi--Campanato gives strict decay. This proves a single uniform terminal estimate. The projected singular passage, the drift cancellation, and the triangular non-singular branch remain valid on one-sided initial and flattened no-flux boundary cylinders.
\end{proof}

\section{Smooth simultaneous continuation and global gluing}\label{sec:global}

Combining Proposition~\ref{prop:affine}, Lemma~\ref{lem:DG}, Lemmas~\ref{lem:meanfree}--\ref{lem:meanflux}, and Theorem~\ref{thm:terminal} yields the local first-exit result.

\begin{theorem}[Local simultaneous continuation]\label{thm:localcontinue}
There exist $R_0>0$ and $\rho_*>0$, depending only on the fixed strong-state and geometric data, such that no smooth approximation with outer augmented energy/dissipation below $\rho_*$ can have a first exit from
\[
u_i<r_A,
\qquad
|\pi_i(u)-\pi_i(v)|<\delta_\tau
\]
inside a face collar. The conclusion holds at interior, initial, and spatial-boundary points.
\end{theorem}

\begin{proof}
Normal exits in the strong face core are excluded by Lemma~\ref{lem:handoff}. While the upper normal branch persists, Lemmas~\ref{lem:meanfree} and \ref{lem:meanflux} give a Dini-summable tangential history. The first annular branch or first scale $k\le\lambda$ is completed by Theorem~\ref{thm:terminal}. Therefore neither a fixed normal nor a fixed tangential first exit is possible.
\end{proof}

Define face collars $U_i=\{v_i<3r_A/4\}$ and an interior region $U_0=\{\min_i v_i>s_b\}$. These sets cover the compact strong trajectory, with at most one face collar singular at a point. On $U_i$ use the Schur chart and the $i$-th perspective corrector; in the overlap Lemma~\ref{lem:handoff} creates a fixed positive lower bound for the regularized normal component before the corrector is switched off.

\begin{theorem}[Smooth simultaneous continuation]\label{thm:smoothcontinue}
For sufficiently small $\lambda$, every smooth admissible approximation with common initial data satisfies
\[
u^{\lambda,n}(t,x)\in\Csafe_{\eta/4}
\qquad\forall(t,x)\in[0,T]\times\overline\Omega,
\]
uniformly in the approximation index.
\end{theorem}

\begin{proof}
Run a first-exit argument for the finite atlas. Face exits are excluded by Theorem~\ref{thm:localcontinue}; transitions to $U_0$ are covered by Lemma~\ref{lem:handoff}; interior exits are excluded by the non-singular triangular compactness mechanism used in Theorem~\ref{thm:terminal}. In the face core $v_j\ge\eta$ for $j\ne i$, whereas $u_i\le r_A$ and the tangential composition error is smaller than $\delta_\tau$. Hence
\[
u_j=(1-u_i)\pi_i(u)_j
\ge(1-r_A)(\eta-\delta_\tau)\ge\eta/2
\qquad(j\ne i)
\]
for the synchronized thresholds. Thus at most one component of $u$ can be below $\eta/2$. The interior compact set gives the same conclusion after a fixed reduction of thresholds.
\end{proof}

The cutoff derivatives are supported only where $r_A/2\le v_i\le3r_A/4$. By Lemma~\ref{lem:handoff}, before this support is reached the regularized state has acquired a fixed lower normal bound. Consequently the perspective density and its first endpoint derivatives are uniformly smooth there, and all cutoff terms are controlled by
\[
\frac c8\Daug_\lambda+C\Eaug_\lambda.
\]
Summing the face estimates and the ordinary interior relative-energy estimate gives the global smooth inequality
\begin{equation}
\Eaug_\lambda(\tau)
+c\int_0^\tau\Daug_\lambda(t)\,dt
\le C\int_0^\tau\Eaug_\lambda(t)\,dt
+C\Flam_\lambda[v;\tau].
\label{eq:globalsmooth}
\end{equation}
Since the initial data agree, $\Eaug_\lambda(0)=0$, and Gronwall yields
\begin{equation}
\esssup_{0\le t\le T}\Eaug_\lambda(t)
+c\int_0^T\Daug_\lambda(t)\,dt
\le C_T\Flam_\lambda[v;T].
\label{eq:gronwall}
\end{equation}

\section{Weak approximation and lower semicontinuity}\label{sec:weak}

We now pass from smooth approximations to admissibly approximable weak solutions.

\begin{proposition}[Weak passage]\label{prop:weak}
For every sufficiently small fixed $\lambda$, an admissibly approximable entropy weak solution satisfies \eqref{eq:safe-main}, and \eqref{eq:globalsmooth} holds with $u^{\lambda,n}$ replaced by $u^\lambda$ for a.e. terminal time.
\end{proposition}

\begin{proof}
The safe-class inclusion from Theorem~\ref{thm:smoothcontinue} is uniform in the approximation index. Since $\Csafe_{\eta/4}$ is closed and $u^{\lambda,n}\to u^\lambda$ a.e., the inclusion passes to the limit.

For fixed $\lambda>0$, the geodesic density is continuous on the compact simplex product. The perspective densities in \eqref{eq:Eaug} are continuous and convex in the regularized endpoint variable and are bounded on the compact state range. Strong $L^2$ convergence plus a.e. convergence therefore gives convergence of the bulk energies; temporal slices are treated by weak lower semicontinuity.

Choose the atlas dissipation explicitly as a finite sum of nonnegative quadratic forms in $\nabla u-\nabla v$ whose coefficients are continuous on the already established safe atlas. The integrand is convex in the gradient variable. State coefficients converge a.e., while the gradients converge weakly in $L^2$. Ioffe lower semicontinuity yields
\[
\int_0^\tau\Daug_\lambda[u^\lambda,v]
\le\liminf_{n\to\infty}
\int_0^\tau\Daug_\lambda[u^{\lambda,n},v].
\]
The forcing is independent of the approximation. Passing to the limit in \eqref{eq:globalsmooth} proves the assertion.
\end{proof}

Theorem~\ref{thm:main} now follows from Proposition~\ref{prop:weak}, \eqref{eq:gronwall}, and Lemma~\ref{lem:statecoerc}.

\section{Interface forcing and quantitative rates}\label{sec:rates}

\begin{proof}[Proof of Theorem~\ref{thm:rate}]
Split each component integral in \eqref{eq:forcing} into $\{v_i\le\lambda\}$ and the dyadic shells
\[
\{2^j\lambda<v_i\le2^{j+1}\lambda\}.
\]
On the first set,
\[
\lambda^2\frac{|\nabla v_i|^2}{(v_i+\lambda)^2}
\le C\lambda^{2\alpha},
\]
while the sublevel-set bound gives measure $O(\lambda^\sigma)$. On shell $j$, the integrand is bounded by
\[
C\lambda^2(2^j\lambda)^{2\alpha-2}
\]
and the shell measure by $C(2^{j+1}\lambda)^\sigma$. Summing the geometric series gives the three regimes in Theorem~\ref{thm:rate}; the region $v_i\ge r_c$ contributes $O(\lambda^2)$.

For $v_i\asymp d^p$ near a smooth hypersurface,
\[
\alpha=\frac{p-1}{p},
\qquad
\sigma=\frac1p,
\qquad
2\alpha+\sigma=2-\frac1p.
\]
Therefore $\rho_\lambda=O(\lambda^{2-1/p})$. Combining this with \eqref{eq:l2main} gives the stated $L^2$ rate.
\end{proof}

\section{Discussion and scope of the result}

The theorem identifies the natural quantitative loss caused by a moving degeneracy. On a fixed face the regularization error may have the order $\lambda^2$, but a vanishing component creates an interface layer whose cost is measured by \eqref{eq:forcing}; its rate depends on the order of contact and the geometry of the zero set.

The proof also clarifies which stronger estimates should not be expected. The normal gradient error is controlled only with the weight dictated by the ratio of regularized normal scales. This is not a technical defect: an unweighted normal coercivity constant vanishes in the boundary layer. Similarly, critical Morrey control of the Schur cross flux is insufficient for pointwise continuation unless it is coupled to the intrinsic normal-amplitude decay. The mean-free/temporal-flux decomposition and the terminal triangular compactness argument are exactly what convert this structure into a summable scale estimate.

The codimension-one hypothesis is essential to the present argument. Higher-codimension corners allow several kernel directions to emerge simultaneously and the resolvent geometry develops additional curvature singularities. Extending the method to such multiple-stratum transitions would require a genuinely different multi-normal analysis and is not claimed here.

\end{document}